\documentclass[12pt,reqno]{amsart}
\usepackage[normalem]{ulem}

\usepackage{multirow}
\usepackage{hhline}
\usepackage{float}
\usepackage{multirow}
\usepackage {multicol}
\usepackage{mathtools}
\usepackage{times}
\usepackage[T1]{fontenc}
\usepackage{mathrsfs}
\usepackage{latexsym}
\usepackage[dvips]{graphics}
\usepackage[titletoc, title]{appendix}
\usepackage{amsmath,amsfonts,amsthm,amssymb,amscd}
\usepackage[dvipsnames]{xcolor}
\usepackage{hyperref}
\usepackage{amsmath}
\usepackage[utf8]{inputenc}

\usepackage{color}
\usepackage{breakurl}

\usepackage{comment}
\newcommand{\bburl}[1]{\textcolor{blue}{\url{#1}}}

\usepackage{caption}

\newtheorem{thm}{Theorem}[section]

\newtheorem{cor}[thm]{Corollary}
\newtheorem{claim}[thm]{Claim}
\newtheorem{lem}[thm]{Lemma}
\newtheorem{prop}[thm]{Proposition}
\newtheorem{exa}[thm]{Example}

\newtheorem{defi}[thm]{Definition}
\newtheorem{rek}[thm]{Remark}

\usepackage[utf8]{inputenc}

\DeclareFixedFont{\ttb}{T1}{txtt}{bx}{n}{12} 
\DeclareFixedFont{\ttm}{T1}{txtt}{m}{n}{12}  

\usepackage{color}
\definecolor{deepblue}{rgb}{0,0,0.5}
\definecolor{deepred}{rgb}{0.6,0,0}
\definecolor{deepgreen}{rgb}{0,0.5,0}

\usepackage{listings}

\definecolor{ao}{rgb}{0.0, 0.5, 0.0}

\numberwithin{equation}{section}

\begin{document}

\title{Generalizing a Pair of Diophantine Equations}

\author[H. V. Chu]{H\`ung Vi\d{\^e}t Chu}
\email{\textcolor{blue}{\href{mailto:hchu@wlu.edu}{hchu@wlu.edu}}}
\address{Department of Mathematics, Washington and Lee University, Lexington, VA 24450, USA}

\author[D. Kim]{Dongho Kim}
\email{\textcolor{blue}{\href{mailto:dkim@as.edu.tw}{dkim@as.edu.tw}}}
\address{Institute of Mathematics, Academia Sinica, Taipei, Taiwan}  

\author[T. Koch]{Theodore Koch}
\email{\textcolor{blue}{\href{mailto:theokoch@umich.edu}{theokoch@umich.edu}}}
\address{Department of Mathematics, University of Michigan, Ann Arbor, MI 48109, USA}

\author[S. J. Miller]{Steven J. Miller}
\email{\textcolor{blue}{\href{mailto:sjm1@williams.edu}{sjm1@williams.edu}},
\textcolor{blue}{\href{mailto:Steven.Miller.MC.96@aya.yale.edu}{Steven.Miller.MC.96@aya.yale.edu}}}
\address{Department of Mathematics, Williams College, Williamstown, MA 01267, USA}

\author[A. V. Nguyen]{Anh Vi\d{\^{e}}t Nguy\~{\^e}n}
\email{\textcolor{blue}{\href{mailto:navgoingcollege2023@tamu.edu}{navgoingcollege2023@tamu.edu}}}
\address{Department of Mathematics, Texas A\&M University, College Station, TX 77840, USA}

\thanks{This work was partially supported by the National Science Foundation DMS2341670. We thank the
participants at Polymath Jr. 2026 for helpful discussions.}



\maketitle

\begin{abstract}
For coprime integers $a$ and $b$, it is known that exactly one of the two Diophantine equations
$$
ax+by\ =\ \frac{(a-1)(b-1)}{2}
\qquad\text{and}\qquad
1+ax+by\ =\ \frac{(a-1)(b-1)}{2}
$$
admits a nonnegative integer solution, and that this solution is unique. We first generalize this result by replacing the right-hand side with an arbitrary integer $m$ and its complement $ab-a-b-m$. This framework enables us to study the existence and uniqueness of nonnegative integer solutions to
$$
ax+by\ =\ \frac{(a-1)(b-1)}{k}
\qquad\text{and}\qquad
1+ax+by\ =\ \frac{(a-1)(b-1)}{k},
$$
where $k$ is a fixed positive integer. We then obtain explicit results when $a$ and $b$ are consecutive Fibonacci numbers. Finally, we examine the original pair of equations in several particular settings, including when $b\equiv \pm1\mod a$, when $b$ is replaced by a higher power, and when the parameters are squared.
\end{abstract}

\tableofcontents

\section{Introduction}

For prime numbers $p$ and $q$, the two Diophantine equations 
\begin{equation}px + qy \ =\ \frac{(p-1)(q-1)}{2}\mbox{ and }1 + px + qy \ =\ \frac{(p-1)(q-1)}{2}\end{equation}
appeared in Beiter's study of cyclotomic polynomials \cite{Be}, which used the property that exactly one of the two equations has a nonnegative integral solution, and furthermore, the solution is unique. Chu \cite{C1} observed that the property holds for every pair of relatively prime $a$ and $b$.

\begin{thm}\cite[Theorem 1.1]{C1}\label{orithm}
Let $a, b\in \mathbb{N}$ be relatively prime. Exactly one of the two equations
\begin{align}
\label{ee1}ax + by &\ =\
\frac{(a - 1)(b - 1)}{2},\\
1 + ax + by &\  =\
\label{ee2}\frac{(a - 1)(b - 1)}{2},
\end{align}
has a nonnegative integral solution, and the solution is unique.
\end{thm}

If either $a$ or $b$ is $1$, we have the right-hand side equal $0$, so we know immediately that \eqref{ee1} has the trivial solution $(x,y) = (0,0)$. We thus consider $a, b\ge 2$ with $\gcd(a,b) = 1$. For convenience, we use the word \textit{solution} to mean an integral solution, and we say an equation is \textit{solvable} if it has a nonnegative solution. 

We define several functions to be used frequently. First, the function $\Gamma: \{(a,b)\in\mathbb{N}^2:\gcd(a,b) = 1\}\rightarrow \{0,1\}$ is defined as
\begin{equation}\Gamma(a,b) \ =\ \begin{cases}0,&\mbox{ if }\eqref{ee1}\mbox{ is solvable},\\ 1,&\mbox{ if }\eqref{ee2}\mbox{ is solvable}.\end{cases}\end{equation}
Second, for each $m\in\mathbb{Z}$, let $\rho(m)\in\{0,\dots,b-1\}$ be the unique integer such that $a\rho(m)\equiv m\mod b$, and set
\begin{equation}\sigma(m)\ :=\ \frac{m-a\rho(m)}{b}.\end{equation}
Third, Arachchi et al. \cite{ACLLMM} used the function \(\Theta(a,b)\), which gives
the unique multiplicative inverse of $a$ modulo $b$ with $1 \le \Theta(a, b) \le b-1$; note that $\Theta(a,b) = \rho(1)$.
If we let 
\begin{equation}\Theta\ :=\ \Theta(a,b)\ =\ \rho(1),\end{equation}
then since $\sigma(1) = (1-a\Theta)/b$, we have
\begin{equation}0\ \le\ -\sigma(1)\ =\ \frac{a\Theta - 1}{b}\ \le\ \frac{a(b-1)-1}{b}\ =\ a -\frac{a+1}{b}\ <\ a,\end{equation}
so 
\begin{equation}\label{e55}0\ \le\ -\sigma(1)\ \le\ a-1.\end{equation}

For brevity, set
\begin{equation}\Gamma\ :=\ \Gamma(a,b)\quad \mbox{ and }\quad \Theta'\ :=\ \Theta(b,a).\end{equation}
We have a characterization of which equation is solvable for a given pair $(a,b)$.
\begin{thm}\cite[Theorem 1]{ACLLMM}\label{master_key}
Let $a, b \ge 2$ with $\gcd(a,b) = 1$. The following are true.
\begin{enumerate}
\item If $a$ is odd, then $\Gamma = 0$ if and only if $\Theta'$ is odd.
\item If $a$ is even, then $\Gamma = 0$ if and only if $\Theta$ is odd.
\end{enumerate}
\end{thm}

Our first result generalizes Theorem \ref{orithm} to a different pair of numbers on the right-hand side. 
\begin{thm}\label{m_exactly1}
Let $a,b\ge 2$ with $\gcd(a,b)=1$, and let $F:=ab-a-b$. For every integer $m$, exactly one of the equations
\begin{equation}\label{ee3}
ax+by\ =\ m
\quad \mbox{ and }\quad
ax+by\ =\ F-m
\end{equation}
is solvable. If additionally, we have $-1\le m\le F+1$, then the nonnegative solution is unique.
\end{thm}
Letting $m = (a-1)(b-1)/2$ in Theorem \ref{m_exactly1} gives Theorem \ref{orithm}. Theorem \ref{m_exactly1} 
reveals an unexpected fact that the existence of a nonnegative solution does not require any restriction on $m$. It is uniqueness that asks for the bound $-1\le m\le F+1$. 
A potentially undesirable feature of \eqref{ee3} is that the right-hand sides, namely $m$ and $F-m$, are not necessarily equal, which motivates us to instead investigate the pair
\begin{equation}ax+by=\frac{(a-1)(b-1)}{k}
\quad \mbox{ and }\quad
1 + ax+by=\frac{(a-1)(b-1)}{k},\end{equation}
where $k$ is a positive integer. Interestingly, Theorem \ref{m_exactly1} facilitates the proof of our second result. 
Let 
\begin{equation}M_k \ :=\ \frac{(a-1)(b-1)}{k}.\end{equation}

\begin{thm}\label{c_atmost1}
Let $a,b\ge 2$ with $\gcd(a,b) = 1$, and let $k\in\mathbb{N}$. Suppose that $k$ divides $(a-1)(b-1)$. Then at most one of the two equations
\begin{equation}
    ax+by \ =\ M_k\quad \mbox{ and }\quad
    1 + ax+by \ =\ M_k
\end{equation}
is solvable. In that case, the nonnegative solution is unique. 
\end{thm}

Inspired by Theorem \ref{master_key}, we would like to characterize the solvability of the two equations in Theorem \ref{c_atmost1}. We see in Lemma \ref{l50} that $ax+by = M_k$ is solvable if and only if $\sigma(M_k)\ge 0$. However, to obtain the divisibility condition as in Theorem \ref{master_key}, we set 
\begin{equation}\label{e51}N_k\ :=\ \frac{k\rho(M_k) - \Theta + 1}{b}.\end{equation}
Our next result states that $ax+by=M_k$ is solvable if and only if $N_k=0$. 

\begin{thm}\label{trichotomy}
    Exactly one of the following occurs. 
    \begin{enumerate}
        \item If $k\rho(M_k) - \Theta + 1 = 0$, the equation $ax+by=M_k$ is solvable, with 
        \begin{equation}(x,y)\ =\ \left(\frac{\Theta-1}{k}, \ \frac{ab-b-a\Theta+1}{bk}\right),\end{equation} while $1+ax+by=M_k$ is not solvable. 
        \item If 
        \begin{equation}\rho(M_k)\ \ge\ \Theta \quad \mbox{ and } \quad \sigma(M_k) \ \ge \ \sigma(1),\end{equation}
        then $1+ax+by=M_k$ is solvable, with 
        \begin{equation}(x,y)\ =\ \left(\rho(M_k)-\Theta, \ \sigma(M_k)-\sigma(1)\right).\end{equation} Meanwhile, the equation $ax+by=M_k$ is not solvable.
        \item Otherwise, neither equation is solvable.
    \end{enumerate}
\end{thm}

\begin{cor}\label{solexist}
The equation $ax+by= M_k$ is solvable if and only if
    \begin{equation}\label{e53}k \mbox{ divides }(\Theta-1) \qquad\text{and}\qquad bk\mbox{ divides } \left(ab-b-a\Theta+1\right).\end{equation}
    In that case, the unique nonnegative solution is
    \begin{equation}\left(\frac{\Theta-1}{k}, \ \frac{ab-b-a\Theta+1}{bk}\right).\end{equation}
\end{cor}
\begin{proof}
    The forward direction follows from Theorem \ref{trichotomy}. We prove the backward direction. 
 Assume \eqref{e53}. Both $(\Theta-1)/k$ and 
 \begin{equation}\frac{ab-b-a\Theta+1}{bk}\ =\ \frac{a-1+\sigma(1)}{k}\end{equation}
 are nonnegative due to \eqref{e55}. We have
\begin{equation}
        a \cdot \frac{\Theta-1}{k}+b\cdot \frac{ab-b-a\Theta+1}{bk} \ =\ \frac{a\Theta-a+ab-b-a\Theta+1}{k}\ =\ M_k.
\end{equation}
Uniqueness follows from Theorem \ref{c_atmost1}.
\end{proof}

\begin{claim} Corollary \ref{solexist} implies Theorem \ref{master_key}.\end{claim}

\begin{proof} Let $k = 2$ in Corollary \ref{solexist}. We consider odd $a$ and even $a$ separately.

\bigskip

\noindent \textbf{Case 1:} $a$ is odd. Due to symmetry, \eqref{e53} can be replaced by 
    \begin{equation}\label{ed60}2 \mbox{ divides }(\Theta'-1) \qquad\text{and}\qquad 2a\mbox{ divides } \left(ab-a-b\Theta'+1\right).\end{equation}
    Note that when $\Theta'$ is odd, 
    \begin{equation}ab-a-b\Theta'+1\ =\ (a-\Theta')b - (a-1)\end{equation}
    is even. Furthermore, 
    \begin{equation}ab-a-b\Theta'+1 \ =\ a(b-1) - (b\Theta'-1)\end{equation}
    is divisible by $a$ due to the definition of $\Theta'$. 
    Therefore, \eqref{ed60} is equivalent to the condition that $\Theta'$ is odd.

\bigskip

\noindent \textbf{Case 2:} $a$ is even. We have \eqref{e53} becomes
    \begin{equation}\label{ed61}2 \mbox{ divides }(\Theta-1) \qquad\text{and}\qquad 2b\mbox{ divides } \left(ab-b-a\Theta+1\right).\end{equation}
    Since $a$ is even, $b$ must be odd, so $(ab-b-a\Theta+1)$ is even. Furthermore, by the definition of $\Theta$, $b$ divides $(ab-b-a\Theta+1)$. Hence, $2b$ divides $(ab-b-a\Theta+1)$. Therefore, \eqref{ed61} is equivalent to the condition that $\Theta$ is odd. 
\end{proof}

Inspired by earlier results \cite{CKKLMYY, C1, CGGJMS, Da, T1}, we apply the above framework to study the equations
\begin{equation}F_n x + F_{n+1}y \ =\ \frac{(F_n-1)(F_{n+1}-1)}{k}\end{equation}
and
\begin{equation} 1 + F_n x + F_{n+1}y\ =\ \frac{(F_n-1)(F_{n+1}-1)}{k}\end{equation}
and demonstrate several solvable and nonsolvable situations. Here $(F_n)_{n=-\infty}^\infty$ is the Fibonacci sequence given by $F_1 = F_2 = 1$ and $F_n = F_{n-1} + F_{n-2}$ for all $n\in\mathbb{Z}$. 

Finally, we compute the function $\Gamma$ and its periodicity in certain special cases including when $b \equiv \pm 1\mod a$, when $b$ is raised to higher powers, and when the parameters are squared. 
\section{Generalizing the right-hand side}

Theorem \ref{m_exactly1}  follows from a natural modification of the proof of Theorem \ref{orithm}. We then use Theorem \ref{m_exactly1} to prove Theorem \ref{c_atmost1} on the solvability dependency between the two equations
\begin{equation}
    ax+by \ =\ M_k\quad \mbox{ and }\quad
    1 + ax+by \ =\ M_k.
\end{equation}
Finally, we characterize their solvability stated in Theorem \ref{trichotomy}.

\begin{proof}[Proof of Theorem \ref{m_exactly1}]
Let $r_1,r_2\in\{0,1,\dots,b-1\}$ satisfy
\begin{equation}
ar_1\ \equiv\ m \mod b\quad\mbox{ and }\quad
ar_2\ \equiv\ F-m\mod b.
\end{equation}
Since $\gcd(a,b)=1$, these are uniquely determined.
Let
\begin{equation}
s_1\ =\ \frac{m-ar_1}{b}\quad \mbox{ and }\quad
s_2\ =\ \frac{F-m-ar_2}{b}.
\end{equation}
Then
\begin{equation}
m\ =\ ar_1+bs_1
\quad\mbox{ and }\quad
F-m\ =\ ar_2+bs_2.
\end{equation}
Adding these two equations gives
\begin{equation}\label{e10}
F\ =\ a(r_1+r_2)+b(s_1+s_2).
\end{equation}
Since $F \equiv -a\mod b$, \eqref{e10} implies that $a(r_1+r_2+1)\equiv0\mod b$.
It follows from $\gcd(a,b)=1$ that $r_1+r_2+1\equiv 0\mod b$.
Since $0\le r_1+r_2\le2b-2$,
the only possibility is
$r_1+r_2=b-1$.
Replacing $r_1 + r_2$ by $b-1$ in \eqref{e10}, we obtain
\begin{equation}ab-a-b\ =\
a(b-1)+b(s_1+s_2).\end{equation}
Thus, $s_1+s_2=-1$. Hence, exactly one of $s_1,s_2$ is nonnegative. Therefore, exactly one of
the equations
\begin{equation}ax+by\ =\ m\quad \mbox{ and }\quad ax+by\ =\ F-m\end{equation}
is solvable.

To prove uniqueness, assume that the equation $ax+by = m$ has two nonnegative solutions, called $(x_1, y_1)$ and $(x_2, y_2)$. Then 
$a|x_1-x_2| = b|y_1 - y_2|$, so $b$ divides $|x_1-x_2|$ because $\gcd(a,b) = 1$. For $i\in \{1,2\}$, we have $by_i\ge 0$, so 
\begin{equation}ax_i\ \le\ m\ \le\ ab-a-b+1,\quad \mbox{i.e.,}\quad x_i\ \le\ \frac{a-1}{a} \cdot(b-1)\ \le\ b-1.\end{equation}
Here we use the assumption $m\le F+1$ above. Hence, $|x_1-x_2|\le b-1$. Since $b$ divides $|x_1-x_2|$, together with $|x_1 - x_2| \le b-1$, we obtain $x_1 = x_2$, so $(x_1, y_1)=(x_2, y_2)$.

The uniqueness proof for $ax+by = F-m$ is similar and uses the assumption $m\ge -1$.
\end{proof}

\begin{cor}\label{k(k-1)cor}
Let $a,b\ge 2$ with $\gcd(a,b) = 1$. Suppose that $k$ divides $(a-1)(b-1)$. Then exactly one of equations
\begin{equation}
    ax+by \ =\ (k-1)M_k \quad \text{ and } \quad 1 + ax+by \ =\ M_k
\end{equation}
is solvable. Furthermore, the nonnegative solution is unique. 
\end{cor}

\begin{proof}
Let 
\begin{equation}F \ =\ ab-a-b\quad \mbox{ and }\quad m \ =\ M_k-1\end{equation}
in Theorem \ref{m_exactly1}. Then 
\begin{equation}-1\ \le\ m\ \le\ (a-1)(b-1) - 1\ =\ F.\end{equation}
Hence, the corollary follows from Theorem \ref{m_exactly1}.
\end{proof}

\begin{cor}\label{k(k-1)cor_2}
Let $a,b\ge 2$ with $\gcd(a,b) = 1$. Suppose that $k$ divides $(a-1)(b-1)$. Then exactly one of equations
\begin{equation}
    ax+by \ =\ M_k \quad \text{ and } \quad 1 + ax+by \ =\ (k-1)M_k
\end{equation}
is solvable. Furthermore, the nonnegative solution is unique. 
\end{cor}
\begin{proof}
   Let 
\begin{equation}F = ab-a-b\quad \mbox{ and }\quad m \ =\ M_k\end{equation}
in Theorem \ref{m_exactly1}. Then 
\begin{equation}0\ \le\ m\ \le\ (a-1)(b-1)\ =\ F+1.\end{equation}
Hence, the corollary follows from Theorem \ref{m_exactly1}.
\end{proof}

\begin{proof}[Proof of Theorem \ref{c_atmost1}]
    Let $m,n\in\mathbb{Z}_{\ge 0}$ be such that $am+bn = M_k$.
    Then $((k-1)m,(k-1)n)$ is a nonnegative solution for 
$ax+by = (k-1)M_k$. By Corollary \ref{k(k-1)cor}, the equation
$1+ax+by = M_k$
    is not solvable. Uniqueness is guaranteed by Corollaries \ref{k(k-1)cor} and \ref{k(k-1)cor_2}.
\end{proof}

We now determine which equation is solvable (if any) when $(a-1)(b-1)/2$ is replaced by $M_k$.

\begin{lem}\label{l50}
    For any integer $m\le (a-1)(b-1)$, the equation $ax+by=m$ is solvable if and only if $\sigma(m)\ge 0$. In that case, $(\rho(m),\sigma(m))$ is the unique nonnegative solution. 
\end{lem}

\begin{proof}
    If $\sigma(m)\ge 0$, then by definition of $\rho(m)$ and $\sigma(m)$, the equation $ax + by = m$ has the nonnegative solution $(\rho(m), \sigma(m))$.

    Conversely, suppose that $ax+by = m$ has a nonnegative solution $(p, q)$. Then $ap\equiv m\mod b$, and 
    \begin{equation}
    ap\ \le\ m\ \le\ ab - a-b+1,\mbox{ so }0\le p\ \le\ b - 1.
    \end{equation}
Hence, $p = \rho(m)$. Thus, we have
\begin{equation}\sigma(m) \ =\ \frac{m-a\rho(m)}{b}\ =\ \frac{m-ap}{b}\ =\ q\ \ge\ 0.\end{equation}

We prove uniqueness. If the equation $ax+by = m$ is solvable, then 
$0\le  m \le ab-a-b+1$. By Theorem \ref{m_exactly1}, the nonnegative solution is unique. 
\end{proof}

\begin{lem} \label{lbound}
    For each integer $k \ge 1$, $N_k$ is an integer in $[0,k-1]$.
\end{lem}
\begin{proof}
    Since $(a-1)(b-1)\equiv 1-a\mod b$, multiplying both sides by $\Theta$ and using $a\Theta\equiv 1\mod b$ give
    \begin{equation}\label{e50}\Theta(a-1)(b-1)\ \equiv\ \Theta(1-a)\ \equiv\ \Theta-1\mod b.\end{equation}
    By the definition of $\rho$, 
    \begin{equation}\frac{(a-1)(b-1)}{k}\ =\ M_k\ \equiv\ a\rho(M_k)\mod b,\end{equation}
    so
    \begin{equation}\label{e51'}\Theta (a-1)(b-1)\ \equiv\ k\rho(M_k)\mod b.\end{equation}
    By \eqref{e50} and \eqref{e51'}, 
    \begin{equation}k\rho(M_k)\ \equiv\ \Theta-1\mod b.\end{equation}
    Therefore, 
    \begin{equation}
    N_k \ = \ \frac{k \rho(M_k) - \Theta + 1}{b}
    \end{equation}
    is an integer.  Since $0\le \rho(M_k)\le b-1$ and $0\le \Theta-1\le b-2$, we get $-1<N_k<k$, so $N_k\in[0, k-1]$.
\end{proof}

\begin{thm}\label{mform}
    The equation $ax+by=M_k$ is solvable if and only if $k\rho(M_k) - \Theta + 1 = 0$. In that case, 
    the unique nonnegative solution is 
    \begin{equation}\left(\frac{\Theta-1}{k}, \ \frac{ab-b-a\Theta+1}{bk}\right).\end{equation}
\end{thm}
\begin{proof}By definition,
$(\rho(M_k), \sigma(M_k))$ is a solution to $ax +by = M_k$.
Thus,
\begin{align}
(a-1)(b-1)&\ = \ ka\rho(M_k) + kb\sigma(M_k)\nonumber\\
&\ =\ a(N_kb + \Theta-1) + kb\sigma(M_k)\nonumber\\
&\ =\ abN_k - b\sigma(1)+1 - a + kb\sigma(M_k).
\end{align}
Therefore,
\begin{equation}ab-b\ =\ abN_k - b\sigma(1) + kb\sigma(M_k),\end{equation}
and dividing both sides by $b$ gives
\begin{equation}\label{e52}a-1\ =\ aN_k - \sigma(1) + k\sigma(M_k).\end{equation}

Assume that $ax+by = M_k$ is solvable. By Lemma \ref{l50}, $\sigma(M_k)\ge 0$, so combining \eqref{e52} with \eqref{e55} gives $a-1 \ge a N_k$, so
\begin{equation}
N_k \ \le \ \frac{a-1}{a} \ < \ 1.
\end{equation}
As $N_k \ge 0$ by Lemma \ref{lbound}, we necessarily have $N_k = 0$. Hence, $k\rho(M_k) - \Theta + 1 = 0$.

Assume that $k\rho(M_k) - \Theta + 1 = 0$. Then \eqref{e52} and \eqref{e55} give  
\begin{equation}
 \sigma(M_k) \ = \ \frac{a-1 + \sigma(1)}{k} \ \ge \ 0.
\end{equation}
 By Lemma \ref{l50}, the equation 
 $ax+by = M_k$ is solvable.

With $N_k = 0$ in \eqref{e51},
we obtain $\rho(M_k) = (\Theta-1)/k$.
Meanwhile, with $N_k = 0$ in \eqref{e52}, we obtain
\begin{equation}\sigma(M_k)\ =\ \frac{a-1+\sigma(1)}{k}\ =\ \frac{ab-b-a\Theta+1}{bk}.\end{equation}
\end{proof}

\begin{thm}\label{wraparound}
    If $\rho(M_k)<\Theta$, then $1+ax+by=M_k$ is not solvable.
\end{thm}
\begin{proof}
    Since
    \begin{equation}
    a \rho(M_k -1) \ \equiv \ M_k - 1 \ \equiv \ a \rho(M_k) - 1 \mod b,
    \end{equation}
    multiplying by \(\Theta \) gives
    \begin{equation} \label{e56}
    \rho (M_k - 1) \ \equiv \ \rho(M_k) - \Theta \mod b.    
    \end{equation}
    Then, because \(0 \le \rho (M_k) < \Theta \le b-1\), we know
    \begin{equation}
    \rho(M_k -1) \ =\ \rho(M_k) - \Theta + b.
    \end{equation}
    We also observe that
    \begin{align}
        \sigma(M_k - 1) \
        &= \ \frac{M_k - 1 - a \rho(M_k - 1)}{b} \nonumber\\
        &= \ \frac{M_k - 1 - a \rho(M_k) + a\Theta - ab}{b} \nonumber\\
        &= \ \frac{M_k - a \rho(M_k)}{b} - \left( \frac{1 - a\Theta}{b} \right) - a\nonumber \\
        &= \ \sigma(M_k) - \sigma(1) - a.
    \end{align}
    Suppose for contradiction that \(1 + ax + by = M_k\) is solvable. By Lemma \ref{l50},
    \begin{equation}
    \sigma(M_k) - \sigma(1) - a \ \ge\ 0.
    \end{equation}
    Combining with \eqref{e55} gives
    \begin{equation}
    \sigma(M_k) \ \ge \ \sigma(1) + a \ \ge \ 1,
    \end{equation}
    so \(ax + by = M_k\) is solvable by Lemma \ref{l50}. As this contradicts Theorem \ref{c_atmost1}, we conclude that \(1 + ax + by = M_k\) is not solvable.
\end{proof}

\begin{thm}\label{m-1form}
    Suppose $\rho(M_k)\ge\Theta$. Then the equation $1+ax+by=M_k$ is solvable if and only if
    \begin{equation} \sigma(M_k) \ \ge \ \sigma(1). \end{equation}
    In that case,  the unique nonnegative solution is
    \begin{equation}
    \left( \rho(M_k) - \Theta, \ \sigma(M_k) - \sigma(1) \right).
    \end{equation}
\end{thm}
\begin{proof}

    Since \(\rho(M_k) - \Theta \in [0,b-1] \), we note from \eqref{e56} that
    \begin{equation}
    \rho(M_k - 1) \ = \ \rho(M_k) - \Theta.
    \end{equation}
    Then,
    \begin{align}
        \sigma(M_k - 1) \
        &= \ \frac{(M_k - 1) - a \rho(M_k-1)}{b} \nonumber\\
        &= \ \frac{M_k - a \rho(M_k) - (1 - a\Theta)}{b} \nonumber\\
        &= \ \sigma(M_k) - \sigma(1).
    \end{align}
    By Lemma \ref{l50}, the equation $1+ax +by = M_k$ is solvable if and only if
    \begin{equation}
    \sigma(M_k) \ \ge \ \sigma(1),
    \end{equation}
    and the unique nonnegative solution is
    \begin{equation}
    \left( \rho(M_k) - \Theta, \ \sigma(M_k) - \sigma(1) \right).
    \end{equation}
\end{proof}

\begin{proof}[Proof of Theorem \ref{trichotomy}]
    Statement (1) follows from Theorems \ref{c_atmost1} and \ref{mform}. Statement (2) follows from Theorems \ref{c_atmost1} and \ref{m-1form}. Statement (3) is immediate from Theorems \ref{mform}, \ref{wraparound}, and \ref{m-1form}. 
\end{proof}

\section{When $(a,b) = (F_n, F_{n+1})$}
Inspired by previous work \cite{CKKLMYY, C1, CGGJMS} that found the unique solution for \eqref{ee1} and \eqref{ee2} when $(a,b)= (F^i_n,F^i_{n+1})$ with $i = 1, 2, 3$, and when $a$ and $b$ are consecutive terms of Fibonacci-like sequences, we are interested in the solvability of
\begin{equation}\label{ed40}F_nx + F_{n+1}y\ =\ \frac{(F_n-1)(F_{n+1}-1)}{k}
\end{equation}
and 
\begin{equation}\label{ed41} 1+F_nx + F_{n+1}y \ =\ \frac{(F_n-1)(F_{n+1}-1)}{k}.\end{equation}

\subsection{Instances of solvability}
We demonstrate that each of the following scenarios can happen: only \eqref{ed40} is solvable; only \eqref{ed41} is solvable; neither equation is solvable. 
\begin{prop} \label{yo}
Let $n\in\mathbb{N}$ be even and let $k\in \mathbb{N}$. If $F_{n-1} \equiv 1 \mod k$, then 
$k$ divides $(F_n-1)(F_{n+1}-1)$. Furthermore, the equation
\begin{equation}
    F_{n}x+F_{n+1}y \ = \ \frac{(F_n-1)(F_{n+1}-1)}{k}
\end{equation}
    has the unique nonnegative solution
\begin{equation}
    x \ = \ y \ = \ \frac{F_{n-1}-1}{k}.
\end{equation}
\end{prop}
\begin{proof}
The first statement follows immediately from the second statement. We prove the second statement. Using Cassini's identity, which  states that
$F_{n+1}F_{n-1}-F_n^2=(-1)^n$ (see \cite[Theorem 5.3]{Koshy2001}),
we have

\begin{align}
    \frac{(F_n-1)(F_{n+1}-1)}{k} &\ = \ \frac{F_nF_{n+1}-F_{n+1}-F_n+1}{k}\nonumber\\
    &\ =\ \frac{(F_{n+1}-F_{n-1})F_{n+1}-F_{n+1}-F_n+1}{k}\nonumber\\
    &\ =\ \frac{F_{n+1}^2 - (F_{n-1}F_{n+1}-1) - F_{n+1}-F_n}{k}\nonumber\\
    & \ = \ \frac{F_{n+1}^2- F_n^2-F_{n+1}-F_n}{k}\nonumber\\
    &\ = \ \frac{(F_{n+1}+ F_n)(F_{n+1}-F_n-1)}{k}\nonumber\\
    &\ = \  F_n\cdot\frac{F_{n-1}-1}{k} + F_{n+1}\cdot\frac{F_{n-1}-1}{k}.
\end{align}
\end{proof}

\begin{prop}\label{mo(n-2)n}
    Let $n\in \mathbb{N}$, and let $k\ge 3$. If  $F_{n} \equiv F_{n+1}\equiv 1 \mod k$, then 
    \begin{equation}
    F_{n}x+F_{n+1}y\ =\ \frac{(F_n-1)(F_{n+1}-1)}{k}
    \end{equation}
    has the unique nonnegative solution
    \begin{equation}
    x\ =\ \frac{F_{n}-1}{k} \quad \text{ and } \quad y\ =\ \frac{F_{n-2}-1}{k}.
    \end{equation}
\end{prop}
\begin{proof} 
Note that $F_1 = F_2 \equiv 1\mod k$, and since $k\ge 3$, the Pisano period\footnote{The \emph{Pisano period} modulo $k$ is the period of the Fibonacci sequence modulo $k$. Its existence follows from the pigeonhole principle applied to the finitely many pairs $(F_n,F_{n+1}) \mod k$; once such a pair repeats, the recurrence determines all subsequent terms periodically. See, for example, Wall~\cite{Wa}.} is even. Hence, the hypothesis implies that $n$ must be odd.
Using Cassini's identity, we have
\begin{align} 
    \frac{(F_n-1)(F_{n+1}-1)}{k} &\ = \  \frac{(F_n-1)(F_{n+1}-1)-F_n(F_n-1)+F_n(F_n-1)}{k}\nonumber\\
    &\ = \  \frac{(F_{n}-1)(F_{n-1}-1)}{k}+\frac{F_n-1}{k}F_n\nonumber\\
    &\ = \ \frac{F_{n}F_{n-1}-F_{n-1}-F_n+1}{k}+\frac{F_n-1}{k}F_n\nonumber\\
    &\ =\ \frac{(F_{n+1}-F_{n-1})F_{n-1}-F_{n-1}-F_n+1}{k} + \frac{F_n-1}{k}F_n\nonumber\\
    &\ = \ \frac{F_{n}^2-F_{n-1}^2-F_{n-1}-F_n}{k}+\frac{F_n-1}{k}F_n\nonumber\\
    & \ = \ \frac{(F_{n}-F_{n-1}-1)(F_{n-1}+F_n)}{k}+\frac{F_n-1}{k}F_n\nonumber\\
    &\ = \ F_n\cdot \frac{F_n-1}{k} + F_{n+1}\cdot \frac{F_{n-2}-1}{k}.
\end{align}
\end{proof}

In Propositions \ref{yo} and \ref{mo(n-2)n}, Equation \eqref{ed40} is solvable; the next proposition gives the values of $n$ such that \eqref{ed41} is solvable when $k = 5$. 
\begin{prop}\label{prop:k-five}
The Diophantine equation
\begin{equation}\label{eq:k-five}
1+F_nx+F_{n+1}y
\ =\ 
\frac{(F_n-1)(F_{n+1}-1)}{5}
\end{equation}
is solvable if and only if
\begin{equation}
n\equiv 7 \mod{20}
\qquad\text{or}\qquad
n\equiv 8 \mod{20}.
\end{equation}

More precisely, if $n\equiv7\mod{20}$, then
\begin{equation}
(x, y)\ =\
\left(\frac{3F_{n+1}-4F_n-1}{5}, \frac{4F_{n-1}-2F_n-1}{5}\right),
\end{equation}
whereas if $n\equiv8\mod{20}$, then
\begin{equation}
(x, y)\ =\
 \left(\frac{4F_n - 2F_{n+1}-1}{5}, \frac{3F_n-4F_{n-1}-1}{5}\right).
\end{equation}
\end{prop}

Proposition \ref{prop:k-five} is a corollary of a more comprehensive result by ChatGPT (GPT-5.6 Sol) \cite{GPT}, which 
correctly characterizes all values of $k$ that make \eqref{ed41} solvable. The authors verify and edit GPT-5.6 Sol's contributions which include Proposition \ref{prop:k-one}, Lemma \ref{lem:floor-fibonacci}, Theorem \ref{thm:admissible-k}, Corollary \ref{ce1}, Corollary \ref{cor:necessary-k}, and Remark \ref{re1} in the next subsection.

We end the current subsection with a situation when neither equation is solvable. To that end, we recall a nonexistence result of Chu \cite{C1}.

\begin{lem}\cite[Lemma 2.1]{C1}\label{chu}
    For $a, b\in\mathbb{N}$ with $\gcd(a,b) = 1$, the equation $ax+by = n$ is not solvable if there exist
    integers $r<b$ and $s < 0$ such that $ar+bs = n$.
\end{lem}

\begin{exa}\normalfont
    The equation $3x + 5y = 4$ has the solution $(x,y) = (3, -1)$. Note that all solutions take the form $(-5n-2, 3n+2)$ with $n\in \mathbb{Z}$, so the equation is not solvable. This is what Lemma \ref{chu} guarantees. 
\end{exa}

\begin{prop}\label{yoyo}
    Let $n\in \mathbb{N}$ and let $k\ge 3$. If $F_{n}\equiv1 \mod k$, and $F_{n+1}\equiv 0 \mod k$, then neither of the equations
    \begin{equation}\label{john}
F_nx+ F_{n+1}y \ = \ \frac{(F_n-1)(F_{n+1}-1)}{k} \quad \text{ and } \quad 1 + F_nx+ F_{n+1}y \ = \ \frac{(F_n-1)(F_{n+1}-1)}{k}
     \end{equation}
is solvable.
\end{prop}
\begin{proof}
Since $F_{-1}\equiv 1\mod k$ and $F_{0}\equiv 0\mod k$, and the fact that the Pisano period modulo $k$ is even, we know that $n$ must be odd. 

    We show the nonexistence of a nonnegative solution to each equation by finding a solution $(r,s)$ where $r<F_{n+1}$ and $s<0$. Then by Lemma \ref{chu}, the claim holds.

    By Cassini's identity, we have
    \begin{align}
    \frac{(F_{n+1}-1)(F_{n}-1)}{k} & \ = \ \frac{F_{n+1}F_{n}-F_{n+1}-F_{n}+1}{k}\nonumber\\
    & \ = \ \frac{F_{n+1}F_{n}-F_{n+1}-F_{n}+(F_{n}^2-F_{n-1}F_{n+1})}{k}\nonumber\\
    & \ = \ \frac{F_{n+2}F_{n}-F_{n+1}-F_{n}-F_{n-1}F_{n+1}}{k}\nonumber\\
    & \ = \ \frac{F_{n+2}-1}{k}F_{n}  +  \frac{-(F_{n-1}+1)}{k}F_{n+1}.
\end{align}
Note that $F_{n+2} = F_{n} + F_{n+1}\equiv 1\mod k$, so $(F_{n+2}-1)/k$ is an integer. Similarly,
$F_{n-1} = F_{n+1} - F_n \equiv - 1\mod k$, so $-(F_{n-1} + 1)/k$ is a negative integer. 
Moreover, 
\begin{equation}\frac{F_{n+2}-1}{k}\ \le\ \frac{F_{n}+F_{n+1}-1}{3}\ <\ F_{n+1}.\end{equation}
Therefore, the equation $F_nx+ F_{n+1}y  = (F_n-1)(F_{n+1}-1)/k$ is not solvable.

Note that  
\begin{align}
F_{n-1}F_n - F_{n-2}F_{n+1}&\ =\ F_{n-1}F_n - F_{n-2}(F_{n-1} + F_n)\nonumber\\
&\ =\ F_nF_{n-3} - F_{n-2}F_{n-1}\ =\ (-1)^{n+1}F_{-1}F_{-2}\ = \ -1,
\end{align}
where the next-to-last equality is due to the generalized Cassini's identity $F_nF_{n+i+j}-F_{n+i}F_{n+j} = (-1)^{n+1} F_iF_j$. Hence, an integral solution to $1 + F_nx + F_{n+1}y = (F_n-1)(F_{n+1}-1)/k$ is 
\begin{equation}\left(\frac{F_{n+2}-1}{k}+F_{n-1}, -\frac{F_{n-1}+1}{k}-F_{n-2}\right).\end{equation}
We have
\begin{align}
\left(\frac{F_{n+2}-1}{k}+F_{n-1}\right) - F_{n+1}&\ \le\ \frac{F_{n+2}-1}{3} - F_n\nonumber\\
&\ =\ \frac{F_{n+2}-1-3F_n}{3}\nonumber\\
&\ =\ \frac{F_{n+1}-1-2F_n}{3}\nonumber\\
&\ =\ \frac{F_{n-1}-1-F_n}{3}\ <\ 0.
\end{align}
By Lemma \ref{chu}, the equation $1 + F_n x + F_{n+1}y = (F_n-1)(F_{n+1}-1)/2$ is also unsolvable.  
\end{proof}

We have for every $m\ge 1$, 
\begin{equation}\label{le1}
F_{2m-1}\ =\ F_m^2 + F_{m-1}^2, \mbox{ so } F_{2m-1}\ \equiv\ F_{m-1}^2 \mod F_m
\end{equation}
(see \cite[Pg.~ 25, Identity (11)]{Vajda1989}).
Moreover, by Cassini's identity, 
\begin{equation}F_{m+1}F_{m-1} - F_m^2 \ =\ (-1)^m, \mbox{ so } (F_{m-1}+F_m)F_{m-1}\ \equiv\ (-1)^m \mod F_m.\end{equation}
Hence, 
\begin{equation}\label{le2} F_{m-1}^2 \ \equiv\ (-1)^m \mod F_m.\end{equation}
From \eqref{le1} and \eqref{le2}, we have $F_{2m-1}\ \equiv\ (-1)^m \mod F_m$.
Choosing $m = 2u$, with $u\ge 1$, we obtain 
\begin{equation}\label{le3}F_{4u-1}\ \equiv\ 1\mod F_{2u}.\end{equation}

We now apply Proposition \ref{yoyo} to determine the solvability of the system
\begin{equation}\begin{cases}F_{4u-1} x + F_{4u}y \ =\ (F_{4u-1}-1)(F_{4u}-1)/F_{2u},\\ 1 + F_{4u-1}x + F_{4u}y\ =\ (F_{4u-1}-1)(F_{4u}-1)/F_{2u}.\end{cases}\end{equation}

\begin{cor}
 The only value of $u\in \mathbb{N}$ that makes the system
 \begin{equation}\begin{cases}F_{4u-1}x + F_{4u} y \ =\ (F_{4u-1}-1)(F_{4u}-1)/F_{2u},\\ 1 + F_{4u-1}x + F_{4u} y \ =\ (F_{4u-1}-1)(F_{4u}-1)/F_{2u}.\end{cases}\end{equation}
 solvable is $u = 1$.
\end{cor}

\begin{proof}
    When $u = 1$, the system becomes
    \begin{equation}\begin{cases}2x + 3 y \ =\ 2,\\ 1 + 2x + 3 y \ =\ 2.\end{cases}\end{equation}
    Then the equation $2x+3y = 2$ has the solution $(1,0)$. 

    For $u \ge 2$, we have $F_{2u}\ge 3$. Furthermore, $F_{4u-1}\equiv 1\mod F_{2u}$ by \eqref{le3}, and $F_{4u} \equiv 0\mod F_{2u}$ because $2u$ divides $4u$. Hence, the hypotheses of Proposition \ref{yoyo} are satisfied. Therefore, the system is not solvable, while $1 + 2x + 3y = 2$ has no solutions.
\end{proof}

\subsection{Admissibility of $k$}
We are interested in the question: for which $k\ge 1$, does there exist an $n\ge 1$ such that \eqref{ed40} or \eqref{ed41}
is solvable? 

First, we consider \eqref{ed40}. If $n=1$ or $n=2$, then $(F_n-1)(F_{n+1}-1) = 0$,
so $x=y=0$ gives a solution of \eqref{ed40}. Hence, every positive integer $k$ makes \eqref{ed40} solvable for some $n\in\mathbb{N}$. We can therefore be more ambitious by asking for infinitely many $n\ge 1$ with solution $x = y > 0$. Still, every integer $k$ satisfies this stronger condition, as stated in the next theorem.

\begin{thm}\label{thm:no-one-all-k}
For every positive integer $k$, there exist infinitely many
$n\ge 3$ for which \eqref{ed40} is solvable with
$x = y > 0$.

More precisely, if $n$ is even and $F_{n-1}\equiv1\mod{k}$,
then
\begin{equation}\label{eq:no-one-xy}
x\ =\ y\ =\ \frac{F_{n-1}-1}{k}
\end{equation}
is a solution of \eqref{ed40}.
\end{thm}

\begin{proof}
See the proof of Proposition \ref{yo}. 
\end{proof}

\begin{defi}\normalfont
A positive integer $k$ is called \emph{admissible} if there exist
$n\ge 1$ and $x,y\in\mathbb Z_{\ge 0}$ satisfying
\eqref{ed41}. 
\end{defi}
The situation changes dramatically when we consider \eqref{ed41}. It turns out that not every $k$ is admissible. 
Subsequent work in this subsection is due to GPT-5.6 Sol \cite{GPT}. The authors have verified that the work is correct and have edited the exposition.

\begin{prop}\label{prop:k-one}
The integer $k=1$ is not admissible.
\end{prop}

\begin{proof}
For $n=1,2$, we have $(F_n-1)(F_{n+1}-1)=0$,
whereas $1+F_nx+F_{n+1}y\ge 1$. Hence, no solution exists in these cases.

Suppose now that $n\ge 3$ and that \eqref{ed41} holds with
$k=1$. Then 
\begin{equation}
1+F_nx+F_{n+1}y
\ =\
(F_n-1)(F_{n+1}-1),\end{equation}
and therefore,
\begin{equation}\label{ed3}
F_n(x+1)+F_{n+1}(y+1)
\ =\
F_nF_{n+1}.
\end{equation}
Since $\gcd(F_n,F_{n+1})=1$, reducing \eqref{ed3} modulo $F_n$ gives
$F_n \mid (y+1)$,
while reducing modulo $F_{n+1}$ gives
$F_{n+1} \mid (x+1)$.
Since $x,y\ge 0$, it follows that
$y+1 \ge F_n$ and 
$x+1 \ge F_{n+1}$.
Consequently,
\begin{equation} F_n(x+1)+F_{n+1}(y+1)\ \ge\ 
F_nF_{n+1}+F_{n+1}F_n,\end{equation}
contradicting \eqref{ed3}.
Thus, $k=1$ is not admissible.
\end{proof}

In view of Proposition~\ref{prop:k-one}, it remains to characterize
the admissible integers $k\ge 2$. Let
$\varphi= (1+\sqrt{5})/2$
be the golden ratio.

\begin{lem}\label{lem:floor-fibonacci}
Let $n\ge 2$ and let $d$ be an integer  satisfying
$1\le d<F_{n+1}$.
Then
\begin{equation}
\left\lfloor d\frac{F_n}{F_{n+1}}\right\rfloor
\ =\
\left\lfloor\frac{d}{\varphi}\right\rfloor.
\end{equation}
\end{lem}

\begin{proof}
We first claim that 
\begin{equation}\label{eq:fib-alpha}
\varphi^{-1}F_{n+1}-F_n
\ =\
(-1)^n\varphi^{-n-1}, \mbox{ for all } n\ge 1.
\end{equation}
We prove \eqref{eq:fib-alpha} by induction. For $n = 1$, the identity holds because $\varphi^2-\varphi-1 = 0$.
For each $n\ge 1$, we have
\begin{align}
\varphi^{-1}F_{n+2}-F_{n+1}&\ =\ \varphi^{-1}(F_{n+1}+F_n)-F_{n+1}\nonumber\\
&\ =\ (\varphi^{-1}-1)F_{n+1} + \varphi^{-1}F_n\nonumber\\
&\ =\ -\varphi^{-2} F_{n+1} + \varphi^{-1}F_n\nonumber\\
&\ =\ -\varphi^{-1}(\varphi^{-1}F_{n+1} - F_n).
\end{align}
Thus, \eqref{eq:fib-alpha} follows by induction.

By \eqref{eq:fib-alpha},
\begin{equation}\label{ed4}\left|d\frac{F_n}{F_{n+1}}-\frac{d}{\varphi}\right|
\ =\
\frac{d}{F_{n+1}}
\left|F_n-\frac{F_{n+1}}{\varphi}\right|\\
\ =\
\frac{d\varphi^{-n-1}}{F_{n+1}} \ <\ \varphi^{-n-1}\ <\ \frac{1}{F_{n+1}},\end{equation}
where the last inequality is due to Binet's formula:
$F_n=(\varphi^n-\varphi^{-n})/\sqrt{5}$ (see \cite[Theorem 5.6]{Koshy2001}).
Write 
\begin{equation}d\frac{F_n}{F_{n+1}}\ =\ q + \frac{r}{F_{n+1}},\end{equation}
for some integers $q\ge 0$ and $r\in\{1,2,\ldots,F_{n+1}-1\}$, which is possible because $dF_n/F_{n+1}$ is not an integer.
Consequently,
\begin{equation}
q+\frac1{F_{n+1}}
\ \le\ d\frac{F_n}{F_{n+1}}\ \le\ 
q+1-\frac1{F_{n+1}}.
\end{equation}
Combining this with \eqref{ed4} gives
\begin{equation}
q\ <\ \frac{d}{\varphi}\ <\ q+1.
\end{equation}
Hence,
\begin{equation}
\left\lfloor \frac{d}{\varphi}\right\rfloor\ =\ q\ =\ \left\lfloor d\frac{F_n}{F_{n+1}}\right\rfloor.
\end{equation}
\end{proof}

We are now ready to characterize all admissible positive integers $k\ge 2$.

\begin{thm}\label{thm:admissible-k}
Let $k\ge 2$, and define
\begin{equation}
r\ =\ r(k)\ :=\
\left\lfloor\frac{k-1}{\varphi}\right\rfloor.
\end{equation}
Then $k$ is admissible if and only if at least one of the following
scenarios occurs.

\begin{enumerate}
    \item There exists $n\ge 1$ such that
    \begin{equation}
    k\ =\ (F_n-1)(F_{n+1}-1).
    \end{equation}
    In this case, the unique solution is $(x,y) = (0,0)$.

    \item There exists an even integer $n\ge 2$ such that 
    \begin{align}
    &k\ <\ F_{n+1},\nonumber\\
    &F_n+rF_{n+1}
    \ \equiv\ -1 \mod k,\mbox{ and }
    \label{eq:even-cong-1}\\
    &rF_n+F_{n+1}
    \ \equiv\ 1 \mod k.
    \label{eq:even-cong-2}
    \end{align}
    In this case, the unique solution is 
    \begin{equation}\label{sol_1}(x,y) \ =\ \left(F_n - \frac{F_n + rF_{n+1}+1}{k}, \frac{rF_n + F_{n+1}-1}{k} - F_{n-1}\right).\end{equation}
    \item There exists an odd integer $n\ge 1$ such that
    \begin{align}
    &k\ <\ F_{n+1},\nonumber\\
        &F_n+(r+1)F_{n+1}
    \ \equiv\ 1 \mod k
    \label{eq:odd-cong-2}, \mbox{ and }\\
    &(r-1)F_n+F_{n+1}
    \ \equiv\ -1 \mod k.
    \label{eq:odd-cong-1}
    \end{align}
    In this case, the unique solution is 
    \begin{equation}\label{sol_2}(x,y) \ =\ \left(\frac{\left(r+1 \right) F_{n+1} +F_n -1}{k} - F_n, F_{n-1} - \frac{F_{n+1} + (r-1)F_n+1}{k} \right).\end{equation}
\end{enumerate}
\end{thm}

\begin{proof}
Assume that $k\ge 2$ is admissible; that is,
\begin{equation}\label{ed400}1 + F_n x + F_{n+1}y\ =\ \frac{(F_n-1)(F_{n+1}-1)}{k}\end{equation}
for some $x, y\in \mathbb{Z}_{\ge 0}$ and $n\ge 1$. 
Multiplying \eqref{ed400} by $k$ and rearranging give
\begin{equation}\label{ed5}
F_n(kx+1)+F_{n+1}(ky+1)\ =\ F_nF_{n+1}-(k-1).
\end{equation}
Then
\begin{equation}
0\ <\ kx+1\ <\ F_{n+1}\quad\mbox{ and }\quad 0\ <\ ky + 1\ <\ F_n.
\end{equation}

\bigskip

\noindent \textbf{Case 1:} $k\ge F_{n+1}$. Since
$0< kx + 1< F_{n+1}\le k$, we must have $x=0$. Similarly, $y = 0$.  Then \eqref{ed5} becomes
\begin{equation}
F_n+F_{n+1}\ =\ F_nF_{n+1}-(k-1).
\end{equation}
Thus,
\begin{equation}
k\ =\ F_nF_{n+1}-F_n-F_{n+1}+1\ =\ (F_n-1)(F_{n+1}-1),
\end{equation}
which is the first scenario.

\bigskip

\noindent \textbf{Case 2:} $k < F_{n+1}$. Since $1\le k-1 < F_{n+1}$, 
Lemma~\ref{lem:floor-fibonacci} yields
\begin{equation}
\left\lfloor (k-1)\frac{F_n}{F_{n+1}}\right\rfloor
\ =\
\left\lfloor\frac{k-1}{\varphi}\right\rfloor.
\end{equation}

\bigskip 
Case 2.1: $n$ is even. Reducing \eqref{ed5} modulo $F_{n+1}$
gives
\begin{equation}F_n(kx+1)\ \equiv\ -(k-1) \mod F_{n+1}.\end{equation}
By Cassini's identity, $F_n^2 \equiv -1\mod F_{n+1}$,
and therefore,
\begin{equation}
kx+1\ \equiv\ (k-1)F_n\mod F_{n+1}.
\end{equation}
Since
$0<kx+1<F_{n+1}$, we use Lemma \ref{lem:floor-fibonacci} to obtain
\begin{equation}\label{ed6}
kx+1\ =\ (k-1)F_n-\left\lfloor(k-1)\frac{F_n}{F_{n+1}}\right\rfloor F_{n+1}\ =\ (k-1)F_n - r F_{n+1}.
\end{equation}
Substituting this into \eqref{ed5} gives
\begin{equation}(k-1)F^2_n-r F_n F_{n+1}+F_{n+1}(ky+1)\ =\ F_nF_{n+1}-(k-1).\end{equation}
Rearranging and substituting $F^2_n + 1 = F_{n-1}F_{n+1}$ gives
\begin{equation}F_{n+1}(ky+1)\ =\ F_nF_{n+1} + r F_n F_{n+1} - (k-1)F_{n-1}F_{n+1}.\end{equation}
Hence,  
\begin{equation}\label{ed7}
ky+1\ =\ \left(r+1\right)F_n-(k-1)F_{n-1}.
\end{equation}
Reducing \eqref{ed6} and \eqref{ed7} modulo $k$ yields
\begin{equation}F_n + r F_{n+1}\ \equiv\ -1\mod k\quad\mbox{ and }\quad rF_n + F_{n+1}\ \equiv\ 1  \mod k.\end{equation}

\bigskip

Case 2.2: $n$ is odd. Reducing \eqref{ed5} modulo $F_{n+1}$ gives
\begin{equation}
F_n(kx+1)\ \equiv\ -(k-1)\mod F_{n+1}.
\end{equation}
By Cassini's identity, $F_n^2 \equiv 1\mod F_{n+1}$.
Therefore,
\begin{equation}kx+1\ \equiv\ -(k-1)F_n\mod F_{n+1},\end{equation}
which, along with $0< kx+1 < F_{n+1}$ and Lemma \ref{lem:floor-fibonacci}, yields
\begin{equation}\label{ed402}kx+ 1 \ =\ \left(\left\lfloor(k-1)\frac{F_n}{F_{n+1}}\right\rfloor+1 \right) F_{n+1} - (k-1)F_n\ =\ \left(r+1 \right) F_{n+1} - (k-1)F_n.\end{equation}
Substituting this into \eqref{ed5}, we obtain
\begin{equation}\left(r+1 \right) F_nF_{n+1} - (k-1)F^2_n+F_{n+1}(ky+1)\ =\ F_nF_{n+1}-(k-1).\end{equation}
Rearranging and using $F_n^2 - 1 = F_{n-1}F_{n+1}$ give
\begin{equation}F_{n+1}(ky+1)\ =\ F_nF_{n+1} + F_{n-1}F_{n+1}(k-1) - \left(r+1 \right) F_nF_{n+1},\end{equation}
so
\begin{equation}\label{ed403}ky + 1\ =\ F_{n-1}(k-1) - rF_n.\end{equation}
We reduce \eqref{ed402} and \eqref{ed403} modulo $k$ to get
\begin{equation}(r+1)F_{n+1}+F_n\ \equiv\  1\mod k\quad \mbox{ and }\quad rF_n+F_{n-1}\ \equiv\ -1\mod k;\end{equation}
that is,
\begin{equation}(r+1)F_{n+1}+F_n\ \equiv\  1\mod k\quad \mbox{ and }\quad (r-1)F_n+F_{n+1}\ \equiv\ -1\mod k.\end{equation}

Conversely, if the first scenario occurs, then 
\begin{equation}\frac{(F_n-1)(F_{n+1}-1)}{k}\ =\ \frac{(F_n-1)(F_{n+1}-1)}{(F_n-1)(F_{n+1}-1)}\ =\ 1,\end{equation}
so the equation $1 + F_n x + F_{n+1} y  = 1$
has the solution $(x,y) = (0,0)$. 

If the second scenario occurs, define the nonnegative integers 
\begin{equation}u\ :=\ F_n - \frac{F_n + rF_{n+1}+1}{k}\quad \mbox{ and }\quad v\ :=\ \frac{rF_n + F_{n+1}-1}{k} - F_{n-1}.\end{equation}
We have
\begin{align}
    &1 + F_n u + F_{n+1} v\nonumber\\
    \ =\ &1 + F_n \left(F_n -\frac{F_n + rF_{n+1}+1}{k}\right) + F_{n+1} \left(\frac{rF_n + F_{n+1}-1}{k}-F_{n-1}\right)\nonumber\\
    \ = \ &\frac{-F_n^2 - F_n + F^2_{n+1} - F_{n+1}}{k}\nonumber\\
    \ =\ &\frac{F_n F_{n+1}-F_n - F_{n+1} + 1}{k}\ =\ \frac{(F_n-1)(F_{n+1}-1)}{k}.
\end{align}

Furthermore,
\begin{align}
F_n + rF_{n+1} + 1&\ =\ F_n + \left\lfloor (k-1)\frac{F_n}{F_{n+1}}\right\rfloor F_{n+1} + 1\nonumber\\
&\ \le\ F_n + (k-1)F_n \ =\ kF_n,
\end{align}
so $u$ is nonnegative. Similarly,
\begin{align}
\frac{rF_n + F_{n+1}-1}{k} - F_{n-1}&\ =\ \frac{ \left\lfloor (k-1)\frac{F_n}{F_{n+1}}\right\rfloor F_{n+1}F_n + F_{n+1}^2 - F_{n+1} - kF_{n-1}F_{n+1}}{F_{n+1} k}\nonumber\\
&\ \ge\ \frac{(k-1)F_n^2 - F_nF_{n+1} + F_{n+1}^2 - F_{n+1} - kF_{n-1}F_{n+1}}{F_{n+1} k}\nonumber\\
&\ =\ \frac{k(F_n^2-F_{n-1}F_{n+1}) - (F_n^2 - F_{n-1}F_{n+1})  - F_{n+1}}{F_{n+1} k}\nonumber\\
&\ =\ \frac{-k + 1 - F_{n+1}}{F_{n+1} k}\nonumber\\
&\ =\ -\frac{1}{F_{n+1}} - \frac{1}{k} + \frac{1}{F_{n+1}k}\nonumber\\
&\ >\ -\frac{1}{F_{3}} - \frac{1}{2}\ =\ -1,
\end{align}
so $v$ is also nonnegative. 

The proof for the third scenario is similar, so we omit it. 
\end{proof}

\begin{cor}\label{ce1}
The integer $k=4$ is not admissible.
\end{cor}

\begin{proof}
Suppose that \(k=4\). Then
$$
r\ =\ \left\lfloor \frac{k-1}{\varphi}\right\rfloor
\ =\ \left\lfloor \frac{3}{\varphi}\right\rfloor\ =\ 1.
$$
We show that none of the three scenarios in
Theorem~\ref{thm:admissible-k} can hold.

First, there is no $n\ge 1$ such that
$4=(F_n-1)(F_{n+1}-1)$.

Next, suppose that \(n\) is even. Since \(r=1\), we have
\eqref{eq:even-cong-1} and \eqref{eq:even-cong-2} become
\[
F_n+F_{n+1}\ \equiv\ -1\mod 4\quad \mbox{ and }\quad
F_n+F_{n+1}\ \equiv\ 1\mod 4,
\]
which cannot happen simultaneously.

Finally, suppose that \(n\) is odd. Then
\eqref{eq:odd-cong-2} and \eqref{eq:odd-cong-1} become
$$F_n+2F_{n+1}\ \equiv\ 1\mod4\quad\mbox{ and }\quad
F_{n+1}\ \equiv\ -1\mod4.$$
Hence,
$$
F_{n}\ \equiv\ 3\mod4
\quad\text{ and }\quad
F_{n+1}\ \equiv\ 3\mod4.
$$
However, the Fibonacci sequence modulo \(4\) is periodic with period
\(6\):
\[
0,1,1,2,3,1,0,1,1,2,3,1,\ldots,
\]
so two consecutive Fibonacci numbers are never both congruent to
$3\mod4$. 

Thus none of the scenarios in Theorem~\ref{thm:admissible-k} holds. Therefore, $4$ is not admissible.
\end{proof}

The characterization in Theorem~\ref{thm:admissible-k} yields a useful necessary condition
depending only on $k$.

\begin{cor}\label{cor:necessary-k}
Let $k\ge 2$ be admissible, and put
\begin{equation}
r\ =\ \left\lfloor\frac{k-1}{\varphi}\right\rfloor.
\end{equation}
If $1 + F_n x+ F_{n+1} y = (F_n-1)(F_{n+1}-1)/k$ is solvable with a solution $(x,y)\neq (0,0)$, then
\begin{equation}\label{eq:necessary-general}
k\mid r(r-2)(r+1)^2.
\end{equation}
\end{cor}

\begin{proof}
Suppose that $1 + F_n x+ F_{n+1} y = (F_n-1)(F_{n+1}-1)/k$ is solvable with a solution $(x,y)\neq (0,0)$.  Hence, either the second or the third scenario in Theorem \ref{thm:admissible-k} occurs.

\bigskip

\noindent \textbf{Case 1:} The second scenario in Theorem \ref{thm:admissible-k} occurs. By
\eqref{eq:even-cong-1} and \eqref{eq:even-cong-2},
\begin{equation}\label{ed406}
F_n+rF_{n+1}\ \equiv\ -1\mod k
\quad \mbox{ and }\quad 
rF_n+F_{n+1}\ \equiv\ 1\mod k.
\end{equation}
For even $n$, Cassini's identity can be written as
\begin{equation}\label{ed407}
F_{n+1}^2-F_nF_{n+1}-F_n^2\ =\ 1.
\end{equation}
By \eqref{ed406} and \eqref{ed407}, we have
\begin{align}
    F_n-F_{n+1}-r^2 + 2\ =\ 
&(1-F_n+rF_n-rF_{n+1})(F_n+rF_{n+1}+1)\nonumber\\
&+(-F_n+F_{n+1}+rF_n)(rF_n+F_{n+1}-1)\nonumber\\
&+(r^2-1)(F_{n+1}^2-F_nF_{n+1}-F_n^2-1)\ \equiv\ 0\mod k,
\end{align}
so 
\begin{equation}\label{ed410}
    F_n - F_{n+1} \ \equiv\ r^2 - 2\mod k.
\end{equation}
Meanwhile, \eqref{ed406} gives
\begin{equation}\label{ed411}
    (r-1)(F_n-F_{n+1})\ \equiv\ 2\mod k.
\end{equation}
Due to \eqref{ed410} and \eqref{ed411}, 
\begin{equation}(r-1)(r^2-2)\ \equiv\ 2\mod k,\end{equation}
which implies \eqref{eq:necessary-general}.

\bigskip

\noindent \textbf{Case 2:} The third scenario in Theorem \ref{thm:admissible-k} occurs. Then
\begin{equation}\label{ed413}
(r-1)F_n+F_{n+1}\ \equiv\ -1\mod k
\end{equation}
and
\begin{equation}\label{ed414}
F_n+(r+1)F_{n+1}\ \equiv\ 1\mod k.
\end{equation}
Multiplying both sides of \eqref{ed413} and \eqref{ed414} by $r+1$ and $r-1$, respectively, we obtain
\begin{equation}\label{ed415}
 (r^2-1)F_n + (r+1)F_{n+1}\ \equiv\ -(r+1)\mod k   
\end{equation}
and
\begin{equation}\label{ed416}
 (r-1)F_n + (r^2-1)F_{n+1}\ \equiv\ r-1\mod k.  
\end{equation}
Subtracting \eqref{ed414} and \eqref{ed415} side by side gives
\begin{equation}\label{ed418}
    (r^2-2)F_n\ \equiv\ -(r + 2)\mod k.
\end{equation}
Subtracting \eqref{ed413} and \eqref{ed416} side by side gives
\begin{equation}\label{ed419}
    (r^2-2)F_{n+1}\ \equiv\ r\mod k.
\end{equation}
By Cassini's identity,
\begin{equation}
F_{n+1}^2-F_nF_{n+1}-F_n^2\ =\ -1,
\end{equation}
so 
\begin{equation}
    ((r^2-2) F_{n+1})^2 - ((r^2-2)F_n)((r^2-2)F_{n+1}) - ((r^2-2)F_n)^2\ =\ -(r^2-2)^2.
\end{equation}
Reducing modulo $k$ and using \eqref{ed418} and \eqref{ed419}, we get
\begin{equation}
    r^2 + r(r+2) - (r+2)^2\ \equiv\ -(r^2-2)^2\mod k. 
\end{equation}
Hence, $k$ must divide $r(r-2)(r+1)^2$.
\end{proof}

\begin{rek}\normalfont
For each fixed $k$, the Fibonacci sequence modulo $k$ is periodic.
Consequently, admissibility of any fixed positive integer $k$ can
be decided by a finite computation.
\end{rek}

\begin{rek}\label{re1}\normalfont
The necessary condition in Corollary~\ref{cor:necessary-k} is not sufficient.
Indeed, for \(k=4\),
\begin{equation}
r\ =\ \left\lfloor \frac{3}{\varphi}\right\rfloor\ =\ 1,
\end{equation}
and hence,
\begin{equation}
r(r-2)(r+1)^2\ =\ -4,
\end{equation}
so
\begin{equation}
4\mid r(r-2)(r+1)^2.
\end{equation}
Thus \(k=4\) satisfies the necessary condition. However, by Corollary \ref{ce1}, $k=4$ is not admissible. Therefore,
\begin{equation}k\mid r(r-2)(r+1)^2\end{equation}
is not sufficient for admissibility.
\end{rek}

\begin{cor}
    The integer $k = 3$ is not admissible.
\end{cor}

\begin{proof}
    Applying Corollary \ref{cor:necessary-k} with $k = 3$, we have $r = 1$, so 
    $k$ does not divide $r(r-2)(r+1)^2$. Hence, \eqref{ed41} is not solvable with a solution $(x,y) \neq (0,0)$. 
However, $(x,y) = (0,0)$ cannot be a solution because there is no $n\ge 1$ such that
    \begin{equation}(F_n-1)(F_{n+1}-1)\ =\ 3.\end{equation}
\end{proof}

\begin{proof}[Proof of Proposition \ref{prop:k-five}]
For $k=5$, we have $r = 2$ in Theorem \ref{thm:admissible-k}.
The first scenario in Theorem \ref{thm:admissible-k} cannot happen when $k = 5$.
We apply the characterization of admissibility according
to the parity of $n$.

\bigskip

\noindent \textbf{Case 1:} $n$ is odd. The corresponding congruences are
\begin{equation}
F_n+F_{n+1}\ \equiv\ 4\mod5\label{eq:k5-odd-1}
\end{equation}
and 
\begin{equation}
F_n+3F_{n+1}\ \equiv\ 1\mod5. \label{eq:k5-odd-2}
\end{equation}
Subtracting \eqref{eq:k5-odd-1} from
\eqref{eq:k5-odd-2} gives $F_{n+1} \equiv 1\mod5$.
Substituting into \eqref{eq:k5-odd-1}, we obtain $F_n\ \equiv\ 3\mod5$.
Thus, $(F_n,F_{n+1})\ \equiv\ (3,1)\mod5$.

The Fibonacci sequence modulo $5$ has Pisano period $20$, and
\begin{equation}
(F_7,F_8)\ =\ (13,21)\ \equiv\ (3,1)\mod5.
\end{equation}
Therefore $(F_n,F_{n+1})\equiv(3,1)\mod5$
if and only if
$n\equiv7\mod{20}$. By \eqref{sol_2}, the solution is 
\begin{equation}\left(\frac{3F_{n+1}-4F_n-1}{5}, \frac{4F_{n-1}-2F_n-1}{5}\right).\end{equation}

\bigskip

\noindent \textbf{Case 2:} $n$ is even. The corresponding congruences are
\begin{equation}
F_n+2F_{n+1}\ \equiv\ 4\mod5 \label{eq:k5-even-1}
\end{equation}
and
\begin{equation}
2F_n+F_{n+1}\ \equiv\ 1\mod5. \label{eq:k5-even-2}
\end{equation}
Multiplying \eqref{eq:k5-even-1} by $2$ and subtracting
\eqref{eq:k5-even-2}, we obtain $3F_{n+1} \equiv 2\mod5$, so 
$F_{n+1}\equiv4\mod5$.
Substituting into \eqref{eq:k5-even-1} gives
$F_n\equiv1\mod5$.
Thus, $(F_n,F_{n+1})\equiv(1,4)\mod5$.

Since
\begin{equation}
(F_8,F_9)\ =\ (21,34)\ \equiv\ (1,4)\mod5,
\end{equation}
and the Pisano period modulo $5$ is $20$, we have $(F_n, F_{n+1})\equiv (1,4)\mod 5$ if and only if
$n\equiv8\mod{20}$. By \eqref{sol_1}, the solution is 
\begin{equation}
    \left(\frac{4F_n - 2F_{n+1}-1}{5}, \frac{3F_n-4F_{n-1}-1}{5}\right).
\end{equation}

\end{proof}


\section{Periodicity and calculation for $\Gamma$}

\subsection{Periodicity of $(\Gamma(a, b^i))_{i=1}^\infty$}

Let
$\mbox{ord}_{a}(b)$ be the multiplicative order of $b$ modulo $a$. 

\begin{thm}\label{pow-per}
    Let $a,b\ge 2$ with $\gcd(a,b)=1$ and odd $a$. Then $(\Gamma(a,b^i))_{i= 1}^\infty$ is periodic with the period dividing $\mbox{ord}_a(b)$.
\end{thm}
 
\begin{proof}
Since $\gcd(a,b) = 1$ and $a,b\ge 2$, neither $a|b^i$ nor $b^i|a$. By Theorem \ref{master_key},
$\Gamma(a, b^i)  = 0$ if and only if $\Theta(b^i, a)$ is odd. 

Let $i\equiv j\mod \mbox{ord}_a(b)$. By definition, $b^{\mbox{ord}_a(b)}\equiv 1 \mod a$, so $b^i\equiv b^{j}\mod a$. Hence, 
    $\Theta(b^i,a)  = \Theta(b^{j},a)$ and thus, $\Gamma(a,b^i)=\Gamma(a,b^j)$.
\end{proof}

\begin{cor}
    Given an odd positive integer $n$ with $n\not\equiv 2\mod 3$, the sequence $(\Gamma(F_{n+1}, F_{n}^i))_{i=1}^\infty$ is either constant or has period $2$. 
\end{cor}

\begin{proof}
    If $n = 1$, then $(\Gamma(F_{n+1}, F_{n}^i))_{i=1}^\infty$ is constantly $0$.  Assume that $n\ge 3$. Let $a = F_{n+1}$ and $b = F_{n}$ in Theorem \ref{pow-per}. Since $n\not\equiv 2\mod 3$, we know that $F_{n+1}$ is odd. Furthermore, from Cassini's identity, $F_{n+1}F_{n-1} - F_n^2 = (-1)^n$, we obtain 
    \begin{equation}F_n^2 \ =\ F_{n+1}F_{n-1} + (-1)^{n+1}\ \equiv\ 1\mod F_{n+1},\end{equation}
    so $\mbox{ord}_{F_{n+1}}(F_n) = 2$. 
\end{proof}

\begin{exa}\normalfont
    The sequence $(\Gamma(F_4, F_{3}^i))_{i=1}^\infty = (\Gamma(3, 2^i))_{i=1}^\infty$ is
    $1, 0, 1, 0, \ldots$. Meanwhile, the sequence $(\Gamma(F_8, F_7^i))_{i=1}^\infty$ is constantly $0$. 
\end{exa}

\subsection{When $b\equiv \pm 1\mod a$}
\begin{thm}\label{ca+1}
    For $a, c\in \mathbb{N}$,
    \begin{equation}\Gamma(a, ca+1) \ =\ 
    \begin{cases} 
        1, & \text{if } a \text{ is even and } c \text{ is odd}, \\ 
        0, & \text{otherwise.} 
    \end{cases}\end{equation}
\end{thm}
 
\begin{proof}
If $a = 1$, then $\Gamma(a, ca+1) = \Gamma(1, c+1) = 0$ because $1\cdot 0 + (c+1)\cdot 0 = (1-1)(c+1-1)/2$. Assume that $a\ge 2$.
    Let $b = ca+1$. Since $\gcd(a,b) = 1$ and $b\equiv 1\mod a$, 
    we have $\Theta'\ =\ 1$,
    which is odd. 

    \bigskip
    
    \noindent \textbf{Case 1:} If $a$ is odd, Theorem~\ref{master_key} gives $\Gamma=0$.

    \bigskip
    
    \noindent \textbf{Case 2:} Suppose that $a$ is even. Multiplying both sides of $ac\equiv-1\mod b$ by $\Theta$ yields
    \begin{equation}
    c\ \equiv\ -\Theta\mod b,\quad \mbox{i.e., } \quad \Theta\ \equiv\ b-c \mod b. 
    \end{equation}
    It follows from $0<c<b$ that
    \begin{equation}
    \Theta\ =\ b-c\ =\ ca+1-c\ =\ c(a-1)+1.
    \end{equation}
    Since $a$ is even, $a-1$ is odd, so $\Theta$
    is odd if and only if $c$ is even. Therefore, by Theorem \ref{master_key}, $\Gamma=0$ if and only if $c$ is even.
\end{proof}

\begin{cor}\label{c1} For odd $n\ge 3$, 
    \begin{equation}\label{e60}\Gamma(F_n^2, F_{n+1}) \ =\ \begin{cases} 
        1, & \text{if } n\equiv 2\mod 3, \\ 
        0, & \text{otherwise.} 
    \end{cases}\end{equation}
\end{cor}
\begin{proof}
    Let $a = F_{n+1}$ and $c = F_{n-1}$ in Theorem \ref{ca+1}. Since $n$ is odd,  Cassini's identity gives
$ca+1 = F_{n-1}F_{n+1} + 1 = F_n^2$.  Then \eqref{e60} follows from Theorem \ref{ca+1} because $2|F_{n+1}$ if and only if $n\equiv 2\mod 3$, in which case, 
$2\nmid F_{n-1}$.
\end{proof}

\begin{thm}\label{ca-1}
    For $a\ge 2$ and $c\ge 2$,
    \begin{equation}
    \Gamma(a, ca-1) \ = \ \begin{cases}
        0, & \text{if } a \text{ is even and } c \text{ is odd}, \\
        1, & \text{otherwise}.
    \end{cases}
\end{equation}
\end{thm}
 
\begin{proof}
Let $b = ca-1$.
Since $\gcd(a,b) = 1$ and $b\equiv -1\mod a$, we have $\Theta' =a-1$.

\bigskip
    
\noindent \textbf{Case 1:} If $a$ is odd, then $a-1$ is even, so Theorem~\ref{master_key} gives $\Gamma=1$.

\bigskip
    
\noindent \textbf{Case 2:} Suppose that $a$ is even. Since $ac\equiv1\mod b$, the multiplicative inverse of $a$ modulo $b$ is represented by $c$. Furthermore,
    \begin{equation}
    b-c\ =\ c(a-1)-1>0,
    \end{equation}
    so $2\le c<b$. Hence, we have $\Theta=c$.
    By Theorem \ref{master_key}, $\Gamma=0$ if and only if $c$ is odd.
\end{proof}

\begin{cor}\label{c2} For even $n\ge 4$, 
\begin{equation}\label{e1}\Gamma(F_{n+1}, F_n^2) \ =\ \begin{cases}0, &\mbox{ if }n \equiv 2\mod 3,\\ 1,&\mbox{otherwise}.\end{cases}\end{equation}
\end{cor}
\begin{proof}
Let $a = F_{n+1}$ and $c = F_{n-1}$ in Theorem \ref{ca-1}. Since $n$ is even,  Cassini's identity gives
$ca-1 = F_{n-1}F_{n+1} - 1 = F_n^2$.  Then \eqref{e1} follows from Theorem \ref{ca-1} because $2|F_{n+1}$ if and only if $n\equiv 2\mod 3$, in which case, 
$2\nmid F_{n-1}$.
\end{proof}

\subsection{A relation between $\Gamma(a, a+i)$ and $\Gamma(a^2, (a+i)^2)$, with $i\in \{1, 2\}$}

We obtain initial results concerning a problem of Chu, Miller, and Tresch \cite[Problem 1.3]{CMT}, which asks for a relation between $\Gamma(a,b)$ and $\Gamma(a^2,b^2)$. In particular, Theorems \ref{thm_square_1} and \ref{thm_square_2} address the cases $b=a+1$ and $b=a+2$, respectively. At present, however, we have not identified a convenient way to determine $\Gamma(a^2,b^2)$ from $\Gamma(a,b)$.

\begin{thm}\label{thm_square_1}
For every integer $a \ge 2$, we have
\begin{equation}\Gamma(a, a+1)\ \neq\ \Gamma(a^2, (a+1)^2).\end{equation}
\end{thm}

\begin{proof}
Since $\gcd(a,a+1)=\gcd(a^2, (a+1)^2)=1$, Theorem \ref{orithm} applies.

First, we consider the equations $\begin{cases}ax + (a+1)y = (a-1)a/2, \\ 1 + ax + (a+1)y = (a-1)a/2.\end{cases}$

\bigskip

If $a$ is odd, then $(a-1)/2$ is a nonnegative integer, and
\begin{equation}
a\cdot \frac{a-1}{2}+(a+1)\cdot 0 \ =\ \frac{(a-1)a}{2},
\end{equation}
so $\Gamma(a,a+1)=0$.

\bigskip

If $a$ is even, then $(a-2)/2$ is a nonnegative integer because $a\ge 2$, and
\begin{equation}
1+a\cdot 0+(a+1)\cdot\frac{a-2}{2}
\ =\ \frac{a^2-a}{2}
\ =\ \frac{(a-1)a}{2},
\end{equation}
so $\Gamma(a,a+1)=1$.

\bigskip

Next, we consider the equations $\begin{cases}a^2 x + (a+1)^2 y  = (a^2-1)((a+1)^2-1)/2,\\ 1 + a^2 x + (a+1)^2 y  = (a^2-1)((a+1)^2-1)/2.\end{cases}$
Note that
\begin{equation}\frac{(a^2-1)((a+1)^2-1)}{2}\ =\ \frac{(a-1)(a+1)a(a+2)}{2}.\end{equation}

\bigskip

If $a$ is odd, then $a\ge 3$ because $a\ge 2$. We have
\begin{align}
1+a^2\cdot\frac{a^2-3}{2}+(a+1)^2\cdot (a-1)&\ =\ \frac{a^4+2a^3-a^2-2a}{2}\nonumber\\
&\ =\ \frac{(a-1)(a+1)a(a+2)}{2},
\end{align} 
so $\Gamma(a^2, (a+1)^2) = 1$.

\bigskip

If $a$ is even, then
\begin{align}
a^2\cdot (a+1) + (a+1)^2\cdot\frac{a(a-2)}{2}&\ =\ a(a+1)\left(a+\frac{(a+1)(a-2)}{2}\right)\nonumber\\
&\ =\ \frac{(a-1)(a+1)a(a+2)}{2}.
\end{align}
Hence, $\Gamma(a^2,(a+1)^2)=0$.
\end{proof}

\begin{thm}\label{thm_square_2}
Let $a\ge 3$ be odd. Then
\begin{equation}\begin{cases}
    \Gamma(a,a+2) \ =\ 0\ \mbox{ and }\ \Gamma(a^2,(a+2)^2)\ =\ 0, &\mbox{ if }a\equiv 1 \mod 8,\\
    \Gamma(a,a+2) \ =\ 1\ \mbox{ and }\ \Gamma(a^2, (a+2)^2)\ =\  1, &\mbox{ if }a\equiv 3\mod 8,\\
    \Gamma(a, a+2) \ =\ 0\ \mbox{ and }\ \Gamma(a^2, (a+2)^2) \ =\ 1, &\mbox{ if }a\equiv 5\mod 8,\\
    \Gamma(a, a+2) \ =\ 1\ \mbox{ and }\ \Gamma(a^2, (a+2)^2) \ =\ 0, &\mbox{ if }a\equiv 7\mod 8.
\end{cases}\end{equation}
\end{thm}

\begin{proof}
Since $\gcd(a,a+2)=\gcd(a^2,(a+2)^2)=1$, Theorem \ref{orithm} applies. Consider the equations
$
\begin{cases}
ax+(a+2)y =   (a-1)(a+1)/2,\\
1+ax+(a+2)y =   (a-1)(a+1)/2.
\end{cases}
$

\bigskip

If $a=8n+1$, then
\begin{equation}(8n+1)\cdot 2n+(8n+3)\cdot 2n
\ =\ 32n^2+8n
\ =\ \frac{8n(8n+2)}2
\ =\ \frac{(a-1)(a+1)}2,\end{equation}
so $\Gamma(a,a+2)=0$.

\bigskip

If $a=8n+3$, then
\begin{align}
1+(8n+3)\cdot (2n+1)+(8n+5)\cdot 2n
&\ =\ 32n^2+24n+4\nonumber\\
&\ =\ \frac{(8n+2)(8n+4)}2
\ =\ \frac{(a-1)(a+1)}2,
\end{align}
so $\Gamma(a,a+2)=1$.

\bigskip
 
If $a=8n+5$, then
\begin{align}
(8n+5)\cdot (2n+1)+(8n+7)\cdot (2n+1)
&\ =\ 32n^2+40n+12\nonumber\\
&\ =\ \frac{(8n+4)(8n+6)}2
\ = \ \frac{(a-1)(a+1)}2,
\end{align}
so $\Gamma(a,a+2)=0$.

\bigskip

If $a=8n+7$, then
\begin{align}
1+(8n+7)\cdot (2n+2)+(8n+9)\cdot (2n+1)
&\ =\ 32n^2+56n+24\nonumber\\
&\ =\ \frac{(8n+6)(8n+8)}2\nonumber\\
&\ =\ \frac{(a-1)(a+1)}2,
\end{align}
so $\Gamma(a,a+2)=1$.

\bigskip

Next, we consider the equations 
$\begin{cases}
a^2x+(a+2)^2y = (a^2-1)((a+2)^2-1)/2,\\
1+a^2x+(a+2)^2y = (a^2-1)((a+2)^2-1)/2.
\end{cases}$
Note that
\begin{equation}
\frac{(a^2-1)((a+2)^2-1)}2
\ =\
\frac{(a-1)(a+1)^2(a+3)}2.
\end{equation}

\bigskip

If $a=8n+1$, then
\begin{align}
a^2\cdot n+(a+2)^2\cdot (32n^2+7n)
 &\ =\ 64n(4n+1)^2(2n+1)\nonumber\\
 &\ =\ \frac{8n(8n+2)^2(8n+4)}{2}\nonumber\\
 &\ =\ \frac{(a-1)(a+1)^2(a+3)}{2},
\end{align}
so $\Gamma(a^2,(a+2)^2)=0$.

\bigskip

If $a=8n+3$, then
\begin{align}
&1+a^2\cdot (16n^2+19n+5)+(a+2)^2\cdot (16n^2+13n+2)\nonumber\\
 &\ =\ 32(4n+1)(2n+1)^2(4n+3)\nonumber\\
 &\ =\ \frac{(8n+2)(8n+4)^2(8n+6)}{2}\nonumber\\
 &\ =\ \frac{(a-1)(a+1)^2(a+3)}{2},
\end{align}
so $\Gamma(a^2,(a+2)^2)=1$.

\bigskip

If $a=8n+5$, then
\begin{align}
   & 1+a^2\cdot (32n^2+55n+23)+(a+2)^2\cdot n\nonumber\\
 &\ =\ 64(2n+1)(4n+3)^2(n+1)\nonumber\\
 &\ =\ \frac{(8n+4)(8n+6)^2(8n+8)}{2}\nonumber\\
 &\ =\ \frac{(a-1)(a+1)^2(a+3)}{2},
\end{align}
so $\Gamma(a^2,(a+2)^2)=1$.

\bigskip

If $a=8n+7$, then
\begin{align}
&a^2\cdot(16n^2+37n+21)+(a+2)^2\cdot(16n^2+27n+11)\nonumber\\
 &\ =\ 128(4n+3)(n+1)^2(4n+5)\nonumber\\
 &\ =\ \frac{(8n+6)(8n+8)^2(8n+10)}{2}\nonumber\\
 &\ =\ \frac{(a-1)(a+1)^2(a+3)}{2},
\end{align}
so $\Gamma(a^2,(a+2)^2)=0$.
\end{proof}


\ \\
\end{document}